\documentclass[12pt,a4paper,reqno]{amsart}
\usepackage{amsfonts,amsthm,amsmath,amssymb,
	amsbsy,euscript,textcomp,mathtools}
\usepackage{stmaryrd}
\usepackage[flushmargin]{footmisc}

\newtheorem{theor}{Theorem}

\theoremstyle{definition}

\newtheorem{proposition}[theor]{Proposition}
\newtheorem{lemma}[theor]{Lemma}

\newtheorem{example}{Example}
\theoremstyle{remark}
\newtheorem{rem}{Remark}

\newcommand{\pinner}{\mathbin{\mathchoice
		{\hbox{\vrule width0.6em depth0pt height0.4pt
				\vrule width0.4pt depth0pt height0.8ex}}
		{\hbox{\vrule width0.6em depth0pt height0.4pt
				\vrule width0.4pt depth0pt height0.8ex}}
		{\hbox{\kern0.14em
				\vrule width0.48em depth0pt height0.4pt
				\vrule width0.4pt depth0pt height0.6ex\kern0.14em}}
		{\hbox{\kern0.1em
				\vrule width0.39em depth0pt height0.4pt
				\vrule width0.4pt depth0pt height0.5ex\kern0.1em}}}}

\newcommand{\cH}{\mathcal{H}}

\newcommand{\cP}{\mathcal{P}}
\newcommand{\cQ}{\mathcal{Q}}

\newcommand{\cX}{{\EuScript X}}    
\newcommand{\cY}{{\EuScript Y}}    
\newcommand{\cZ}{{\EuScript Z}}    

\newcommand{\lshad}{[\![}
\newcommand{\rshad}{]\!]}

\newcommand{\schouten}[1]{\lshad {#1} \rshad}

\newcommand{\by}[1]{\textit{{#1}}}
\newcommand{\jour}[1]{\textit{{#1}}}
\newcommand{\vol}[1]{\textbf{{#1}}}
\newcommand{\book}[1]{\textrm{{#1}}}

\title[The Kontsevich tetrahedral flow in dimension $n=2$]{The Kontsevich tetrahedral flow in 2D: a toy model}

\author{Anass Bouisaghouane}

\thanks{\textit{Address}: Johann Ber\-nou\-lli Institute for Mathematics and Computer Science, University of Groningen,
	P.O.~Box 407, 9700~AK Groningen, The Netherlands. 
	\quad${}^{*}$\textit{E-mail}: \texttt{a.bouisaghouane.1\symbol{"40}student.rug.nl}}

\date{\today}



\begin{document}
\begin{abstract}
In the paper ``Formality conjecture'' (1996) Kontsevich designed a universal flow $\dot{\cP}=\cQ_{a:b}(\cP)=a\Gamma_{1}+b\Gamma_{2}$ on the spaces of Poisson structures $\cP$ on all affine manifolds of dimension $n \geqslant 2$. We prove a claim from {\it{loc.\,cit.\,}} stating that if $n=2$, the flow $\cQ_{1:0}=\Gamma_{1}(\cP)$ is Poisson-cohomology trivial: $\Gamma_{1}(\cP)=\schouten{\cP,\cX}$ for some vector field~$\cX$; we examine the structure of the space of solutions $\cX$. Both the construction of differential polynomials $\Gamma_{1}(\cP)$ and $\Gamma_{2}(\cP)$ and the technique to study them remain valid in higher dimensions $n \geqslant 3$, but neither the trivializing vector field $\cX$ nor the setting $b:=0$ survive at $n\geqslant 3$, where the balance is $a:b=1:6$.
\end{abstract}
\maketitle
\subsection*{Introduction}
\noindent
Let $\cP=(\cP^{ij})$ be a Poisson bi-vector (whose coefficient matrix is skew-symmetric and satisfies the Jacobi identity\footnote{The Schouten bracket $\schouten{\cdot,\cdot}$ is a unique extension of the commutator $[\cdot,\cdot]$ on the space of vector fields  $\cX^{1}(N^{n})$ to the space $\cX^{k}(N^{n})$ of polyvector fields. By definition, the Schouten bracket coincides with the Lie bracket when evaluated on $1$-vectors. When evaluated on p-vector $\cX$,  q-vector $\cY$ and r-vector $\cZ$, the Schouten bracket satisfies the equations $\schouten{\cX,\cY}=-(-1)^{(p-1)(q-1)}\schouten{\cY,\cX}$ and $\schouten{\cX,\cY\wedge\cZ}=\schouten{\cX,\cY}\wedge\cZ+(-1)^{(p-1)q}\cY\wedge\schouten{\cX,\cZ}$.} $\schouten{\cP,\cP}=0$) on a real affine $n$-dimensional manifold~$N^{n}$ and denote by $\textbf{x}=(x^{1},\ldots, x^{n})$ local coordinates. In \cite{Ascona96}, Kontsevich introduced two differential polynomials\footnote{The second differential polynomial $\Gamma_{2}^{im}$ is not skew-symmetric in $(i,m)$ so that one must skew-symmetrize in $(i,m)$: $\dot{\cP}_{2}^{im}(\omega,\eta)=\frac{1}{2}\left(\Gamma_{2}^{im}(\cP)-\Gamma_{2}^{mi}(\cP)\right)\partial_{x^{i}}(\omega)\wedge \partial_{x^{m}}(\eta)$, in order to construct a similar flow as the one obtained from $\Gamma_{1}$.
}
in the coefficients $\cP^{\alpha\beta}$ and their derivatives $\cP^{\alpha\beta}_{\sigma}\stackrel{\text{def}}{=} \partial ^{|\sigma|} / \partial x^{\sigma} (\cP^{\alpha\beta})$, where $\sigma=(\sigma_{1}\cdots\sigma_{k})$, $k \in \mathbb{N}$, denotes the indices of the coordinates with respect to which we differentiate, e.g. $\partial^{|(122)|}/\partial x^{(122)}=\partial^{3}/\partial x^{1}x^{2}x^{2}$:
\begin{align*}
\raisebox{0pt}[10pt][10pt]{$ 
\Gamma_{1}^{ij}(\cP) = \sum\limits_{k,\ell,m,k',\ell',m'=1}^{n} \cP^{ij}_{klm} \cP^{kk'}_{\ell'} \cP^{\ell\ell'}_{m'} \cP^{mm'}_{k'}, \quad
\Gamma_{2}^{im}(\cP) = \sum\limits_{j,k,\ell,k',\ell',m'=1}^{n} \cP^{ij}_{k\ell} \cP^{km}_{k'\ell'} \cP^{k'\ell}_{m'} \cP^{m'\ell'}_{j}.
$}
\end{align*}
%
From any initial bi-vector $\cP$, the coefficients of a new bi-vector are constructed using the differential polynomial $\Gamma_{1}$: $\dot{\cP}_{1}^{ij}(\omega,\eta)=\Gamma^{ij}_{1}(\cP)\partial_{x^{i}}(\omega)\wedge \partial_{x^{j}}(\eta)$ where the functions $\omega$, $\eta \in C^\infty(N^n)$. We thereby obtain a `flow' on the space of bi-vectors with infinitesimal deformation $\cP \mapsto \cP+\epsilon \Gamma_{1}(\cP)+\bar{o}(\epsilon),\ \epsilon \in \mathbb{R}$. 
%
From a Poisson bi-vector $\cP$, we construct the classical Poisson differential $\partial_{\cP}=\schouten{\cP,\cdot}$ 
and obtain the Poisson complex:
$$0 \to \Bbbk \xhookrightarrow{} C^\infty(N^n) \to \cX^{1}(N^{n}) \to \cX^{2}(N^{n}) \to \cdots \to \cX^{n}(N^{n}) \to 0. $$
Does either of the flows mark a $\partial_{\cP}$-\/trivial class in the Poisson cohomology? If $n=2$, every flow is $\partial_{\cP}$-\/closed because the Schouten bracket of the bi-vector $\cP$ with the bi-vector $\Gamma_{1}(\cP)$ or $\Gamma_{2}(\cP)$ is a tri-vector that sure vanishes on a two-dimensional affine manifold. 
The property we explore in this paper is the exactness, with respect to the Poisson differential, of the Kontsevich tetrahedral flows over 2-dimensional affine manifolds.
\vspace{-3mm}
\begin{center}
\rule{110mm}{1pt}
\end{center}

We first expand both the bi-vectors obtained from $\Gamma_{1}$ and skew-sym\-met\-rized $\Gamma_{2}$ with respect to the local coordinates $x^{1}=x,x^{2}=y$ 
of our 2-dimensional manifold. This means that we can expand the differential polynomials with their indices $i,j,k,\ell,m,k',\ell',m' \in \{1,2\}$. In dimension $n=2$, every bi-vector has only one unique component, $\cP^{12}$, which is denoted by $u$. 
\begin{proposition}
The only non-zero component of the bi-vector flow $\dot{\cP}_{1}$ is expressed in terms of $u,x$ and $y$ via
\begin{align}
\label{Flow2D}
\dot{\cP}_{1} &= u_{xxx}u_y^{3}-u_{yyy}u_x^{3}-3u_{xxy}u_{x}u_{y}^{2}+3u_{xyy}u_x^{2}u_{y}.
\end{align}
The bi-vector flow $\dot{\cP}_{2}$ vanishes identically (in dimension $n=2$ under study).
\end{proposition}
\begin{proof}
We expand the differential formula for $\Gamma_{1}^{12}(\cP)$. Taking into account that the bi-vector coefficient $\cP^{ij}$ can be non-zero only when $i=1,j=2$ or $i=2,j=1$, the sum on the right-hand side of the formula for $\Gamma_{1}^{12}(\cP)$ yields:
\small
\begin{eqnarray*}
\Gamma_{1}^{12}(\cP) =& \hspace{-18mm}  \sum\limits_{k,\ell,m, k',\ell',m'=1}^{n} \cP^{12}_{klm} \cP^{kk'}_{\ell'} \cP^{\ell\ell'}_{m'} \cP^{mm'}_{k'}
=\sum\limits_{\ell,m,\ell',m'=1}^{n} \cP^{12}_{1lm} \cP^{12}_{\ell'} \cP^{\ell\ell'}_{m'} \cP^{mm'}_{2}+\cP^{12}_{2lm} \cP^{21}_{\ell'} \cP^{\ell\ell'}_{m'} \cP^{mm'}_{1}\\
=& \sum\limits_{m,m'=1}^{n} \cP^{12}_{11m} \cP^{12}_{2} \cP^{12}_{m'} \cP^{mm'}_{2}+ \cP^{12}_{12m} \cP^{12}_{1} \cP^{21}_{m'} \cP^{mm'}_{2}+ \cP^{12}_{21m} \cP^{21}_{2} \cP^{12}_{m'} \cP^{mm'}_{1}+\cP^{12}_{22m} \cP^{21}_{1} \cP^{21}_{m'} \cP^{mm'}_{1}\\
=& \hspace{-18mm}\cP^{12}_{111} \cP^{12}_{2} \cP^{12}_{2} \cP^{12}_{2}+\cP^{12}_{222} \cP^{21}_{1} \cP^{21}_{1} \cP^{21}_{1}+3\cP^{12}_{112} \cP^{12}_{2} \cP^{12}_{2} \cP^{21}_{2}+3\cP^{12}_{122} \cP^{12}_{2} \cP^{21}_{1} \cP^{21}_{2}
\end{eqnarray*}
\normalsize
A similar computation for $\dot{\cP}_{2}$ is given in Appendix~\ref{App}.
\end{proof}
\noindent
From now on, we denote the bi-vector $\dot{\cP}_{1}$ obtained from $\Gamma_{1}$ unambiguously by $\dot{\cP}$.

The definition of a flow $\dot{\cP}$ to be $\partial_{\cP}$-exact is that there exists a vector field $\cX$, defined globally on $N^{n}$, whose coefficients are differential polynomials in $\cP$, such that the flow is the Schouten bracket of this vector field $\cX$ with the Poisson bi-vector
: $\dot{\cP}=\partial_{\cP}\cX=\schouten{\cP,\cX}$.
It is readily seen that for a two-dimensional vector field $\cX=F\,\partial/\partial x + G\,\partial/\partial y$, we have that the only non-zero component of $\schouten{\cP,\cX}$ is $\schouten{\cP,\cX}^{12}= u_{x}F+u_{y}G-u(F_{x}+G_{y})$.
The conjugation equation $\dot{\cP}=\schouten{\cP,\cX}$ is now expressed via:
\begin{align}
\label{ConEq}
u_{xxx}u_y^{3}-u_{yyy}u_x^{3}-3u_{xxy}u_{x}u_{y}^{2}+3u_{xyy}u_x^{2}u_{y}=u(F_{x}+G_{y})-u_{x}F-u_{y}G.
\end{align}
Let us examine using Jets \cite{Jets} whether equation~\eqref{ConEq} has any globally defined solutions $\cX$.
\begin{example}\label{eg1}
A solution $\cX=F\,\partial/\partial x + G\,\partial/\partial y$ to equation~\eqref{ConEq}, depending on the differential monomials $u_{\sigma}$ not exceeding order $|\sigma|=3$, is:
\begin{align}\label{sol1}
\begin{split}
F&=\phantom{-}u_{yyy}u_{x}^{2}-{2}u_{xyy}u_{x}u_{y}+u_{xxy}u_{y}^{2}+{2}u_{xx}u_{yy}u_{y}-{2}u_{xy}^{2}u_{y}+cu_{y},\\
G&=-u_{xxx}u_{y}^{2}+{2}u_{xxy}u_{x}u_{y}-u_{xyy}u_{x}^{2}-{2}u_{xx}u_{yy}u_{x}+{2}u_{xy}^{2}u_{x}-cu_{x},\ c \in \mathbb{R}.
\end{split}
\end{align}
It is easily verified that this vector field is divergence-free\footnote{If the vector field $\cX$ is divergence-free, then equation~\eqref{ConEq} with respect to $F$ and $G$ splits into an algebraic equation and an additional restriction on the vector field components. The solution to the divergence-free equation is of the form $F =\varphi(x-y)-\psi(y)$ and $G=\varphi(x-y)-\chi(x),$ for functions $\varphi$ ,$\psi$ and $\chi$ to be determined.}: $F_{x}+G_{y}=0$.
\end{example}
One can obtain a more general solution by allowing higher differential orders $|\sigma|$ of the monomials $u_{\sigma}$ that make up the vector field components $F$ and $G$. The following method was hinted to us by P. Safronov and I. Khavkine.\footnote{\label{constants}See \text{http://mathoverflow.net/questions/209376}. A solution reported there is obtained by setting the constant $c=0$ in equation~\eqref{sol1}. Choosing constants $a,b,c,d,e$ equal to zero in solution \eqref{EqBigSol} from Proposition~\ref{SolGen} yields again that solution.} Consider the weight homomorphism $\text{wt}_{x}(u_{\sigma})=(\#x\in \sigma)$, such that $\text{wt}_{x}(u_{\sigma}u_{\tau})=\text{wt}_{x}(u_{\sigma})+\text{wt}_{x}(u_{\tau})$. Similarly, we let $\text{wt}_{y}(u_{\sigma})=(\#y\in \sigma)$. We now let the differential polynomials that make up either side of equation~\eqref{ConEq} be quartic in $u$ and homogeneous in $\text{wt}_{x}$ and $\text{wt}_{y}$. Under these assumptions, it follows from comparing the left-hand and righ-hand sides of equation~\eqref{ConEq} that the polynomial terms of $F$ and $G$ are cubic in $u_{\sigma}$, that the differential polynomials of $F$ have $\text{wt}_{x}=2$ and $\text{wt}_{y}=3$, and those of $G$ have $\text{wt}_{x}=3$ and $\text{wt}_{y}=2$.
Specifically, $F=\sum_{i=1}^{N}c_{i}u_{\sigma_{1}^{i}}u_{\sigma_{2}^{i}}u_{\sigma_{3}^{i}}$ such that $\text{wt}_{x}(u_{\sigma_{1}^{i}}u_{\sigma_{2}^{i}}u_{\sigma_{3}^{i}})=2$ and $\text{wt}_{y}(u_{\sigma_{1}^{i}}u_{\sigma_{2}^{i}}u_{\sigma_{3}^{i}})=3$ for all $1\leqslant i \leqslant N$, where $N\in \mathbb{N}$ and $c_{i}\in \mathbb{R}$. Similarly, $G=\sum_{j=1}^{M}d_{j}u_{\tau_{1}^{j}}u_{\tau_{2}^{j}}u_{\tau_{3}^{j}}$ such that $\text{wt}_{x}(u_{\tau_{1}^{j}}u_{\tau_{2}^{j}}u_{\tau_{3}^{j}})=3$ and $\text{wt}_{y}(u_{\tau_{1}^{j}}u_{\tau_{2}^{j}}u_{\tau_{3}^{j}})=2$ for all $1\leqslant j \leqslant M$, where $M\in \mathbb{N}$ and $d_{j}\in \mathbb{R}$. In fact, $M=N=12$.
We now substitute these polynomial expressions with undetermined coefficients for $F$ and $G$ into equation~\eqref{ConEq} and solve for the coefficients.
%
%
\begin{proposition}\label{SolGen}
Under all the above assumptions, the space of solutions to equation~\eqref{ConEq} is given by the vector fields $\cX=F\,\partial/\partial x + G\,\partial/\partial y$, where
\begin{align}\label{EqBigSol}
\begin{split}
&F=(a+b)u_{xxyyy}u^{2}+bu_{xxyy}u_{y}u+cu_{xyyy}u_{x}u+u_{xxy}u_{y}^{2}+du_{xxy}u_{yy}u+eu_{xyy}u_{xy}u\\
&\qquad+(a+d)u_{yyy}u_{xx}u-2u_{xyy}u_{x}u_{y}+u_{yyy}u_{x}^{2}+2u_{xx}u_{yy}u_{y}-2u_{xy}^{2}u_{y},
\\
&G=-(a+b)u_{xxxyy}u^{2}+au_{xxxy}u_{y}u-(a+b+c)u_{xxyy}u_{x}u-(a+d)u_{xxx}u_{yy}u\\
&\qquad+(a+c-e)u_{xxy}u_{xy}u-(a+c+d)u_{xyy}u_{xx}u-u_{xxx}u_{y}^{2}+2u_{xxy}u_{x}u_{y}-u_{xyy}u_{x}^{2}\\&\qquad-2u_{xx}u_{yy}u_{x}+2u_{xy}^{2}u_{x},
\end{split}
\end{align}
and $a,b,c,d,e$ are real constants. These vector fields are not divergence-free except for the case when all coefficients vanish (without the entire solution vanishing, because the solution then becomes equal to \eqref{sol1} at $c=0$). 
\end{proposition}
Now that we have a vector field $\cX$ such that $\partial_{\cP}\cX=\dot{\cP}$ over a 2-dimensional affine manifold, we inspect 
whether its construction can be repeated to trivialize $\dot{\cP}$ on all finite-dimensional affine manifolds. For this, we pass from the dimension dependent differential polynomials to dimension independent representations of said differential polynomials, by using Kontsevich graphs.

The graphs consist of ground vertices, drawn at the bottom of the figure, and internal vertices.
All internal vertices in the graph represent a copy of $\cP^{ij}$. Every vertex is the tail for an ordered pair of outgoing edges. These edges are labelled by the summation indices $i,j$ in $\cP^{ij}$ inside the vertex. In the 2-dimensional case, every vertex with $u=\cP^{12}$ has outgoing edges labelled $1$ and $2$, respectively. Incoming edges with value $1$ or $2$ imply differentiation of the target vertex with respect to $x^{1}$ or $x^{2}$, respectively. Finally, the internal vertices connected by edges form a product of bi-vector components $u=\cP^{12}$ and derivatives thereof. The result is then multiplied by the arguments of the ground vertices that are differentiated as specified by the incoming edges.
\begin{lemma}\label{SolGraph}
Solving equation~\eqref{ConEq}, the divergence-free vector field \eqref{sol1} in Example~\ref{eg1} is realizable in terms of the Kontsevich graphs:
\begin{align}
\label{EqXFig}
\cX \, =
\hspace{-20mm}
\raisebox{-30pt}[30pt][15pt]{
	\unitlength=1.mm
	\special{em:linewidth 0.4pt}
	\linethickness{0.4pt}
	\begin{picture}(30.00,27.00)
	\put(24.7,10.65){\vector(0,-1){6.00}}
	\put(21.7,7.65){\makebox(0,0)[rc]{}}
	\qbezier(24.7,10.65)(14.7,6.65)(18.7,13.65)
	\put(18.7,13.65){\vector(1,1){0.00}}
	\put(27.7,7.65){\makebox(0,0)[lc]{}}	
	\put(19.7,15.65){\vector(1,-1){5.00}}
	\put(21.7,12.0){\makebox(0,0)[rc]{}}
	\put(29.7,15.65){\vector(-1,-1){5.00}}
	\put(27.7,12.0){\makebox(0,0)[lc]{}}
	\qbezier(19.7,15.65)(24.7,19.65)(29.7,15.65)
	\put(29.7,15.65){\vector(2,-1){0.00}}
	\put(25.3,18.965){\makebox(0,0)[rc]{}}
	\qbezier(19.7,15.65)(24.7,11.65)(29.7,15.65)
	\put(19.7,15.65){\vector(-2,1){0.00}}
	\put(25.3,14.965){\makebox(0,0)[rc]{}}
	\put(24.7,15.65){\oval(12,12)}
	\end{picture}
}
\,+c\cdot
\hspace{-25mm}
\raisebox{-30pt}[10pt][0pt]{
	\unitlength=1.mm
	\special{em:linewidth 0.4pt}
	\linethickness{0.4pt}
	\begin{picture}(30.00,27.00)
	\put(24.7,14.15){\circle{7}}
	\put(24.7,10.65){\vector(1,0){0.00}}
	\put(21.7,7.65){\makebox(0,0)[rc]{}}
	\put(24.7,10.65){\vector(0,-1){6.00}}
	\put(27.7,7.65){\makebox(0,0)[lc]{}}	
	\end{picture}
}.
\end{align}
\end{lemma}
\begin{proof}
Whenever $F_{x}+G_{y}=0$ on $\mathbb{R}^{2} \ni (x,y)$, the vector field components $F$ and $G$ can be potentiated such that $F=H_{y},G=-H_{x}$; we let  
$$H=u_{xx}u_{y}^{2}-2u_{xy}u_{x}u_{y}+u_{yy}u_{x}^{2}+cu,$$
where $H$ can be viewed as a scalar function.
These differential monomials can be written as Kontsevich graphs (see above):
$$\cH=
\hspace{-20mm}
\raisebox{-30pt}[25pt][17pt]{
	\unitlength=1.mm
	\special{em:linewidth 0.4pt}
	\linethickness{0.4pt}
	\begin{picture}(30.00,27.00)
	\put(24.7,10.65){\vector(1,-2){3.00}}
	\put(21.7,7.65){\makebox(0,0)[rc]{\tiny$1$}}
	\put(24.7,10.65){\vector(-1,-2){3.00}}
	\put(27.7,7.65){\makebox(0,0)[lc]{\tiny$2$}}	
	\put(19.7,15.65){\vector(1,-1){5.00}}
	\put(21.7,12.0){\makebox(0,0)[rc]{\tiny$1$}}
	
	\put(29.7,15.65){\vector(-1,-1){5.00}}
	\put(27.7,12.0){\makebox(0,0)[lc]{\tiny$1$}}
	\qbezier(19.7,15.65)(24.7,19.65)(29.7,15.65)
	\put(29.7,15.65){\vector(2,-1){0.00}}
	\put(25.3,18.965){\makebox(0,0)[rc]{\tiny$2$}}
	
	\qbezier(19.7,15.65)(24.7,11.65)(29.7,15.65)
	\put(19.7,15.65){\vector(-2,1){0.00}}
	\put(25.3,14.965){\makebox(0,0)[rc]{\tiny$2$}}
	\end{picture}
}
-
\hspace{-22mm}
\raisebox{-30pt}[10pt][10pt]{
	\unitlength=1.mm
	\special{em:linewidth 0.4pt}
	\linethickness{0.4pt}
	\begin{picture}(30.00,27.00)
	\put(24.7,10.65){\vector(1,-2){3.00}}
	\put(21.7,7.65){\makebox(0,0)[rc]{\tiny$1$}}
	\put(24.7,10.65){\vector(-1,-2){3.00}}
	\put(27.7,7.65){\makebox(0,0)[lc]{\tiny$2$}}	
	\put(19.7,15.65){\vector(1,-1){5.00}}
	\put(21.7,12.0){\makebox(0,0)[rc]{\tiny$2$}}
	
	\put(29.7,15.65){\vector(-1,-1){5.00}}
	\put(27.7,12.0){\makebox(0,0)[lc]{\tiny$1$}}
	\qbezier(19.7,15.65)(24.7,19.65)(29.7,15.65)
	\put(29.7,15.65){\vector(2,-1){0.00}}
	\put(25.3,18.965){\makebox(0,0)[rc]{\tiny$1$}}
	
	\qbezier(19.7,15.65)(24.7,11.65)(29.7,15.65)
	\put(19.7,15.65){\vector(-2,1){0.00}}
	\put(25.3,14.965){\makebox(0,0)[rc]{\tiny$2$}}
	\end{picture}
}
-
\hspace{-22mm}
\raisebox{-30pt}[10pt][10pt]{
	\unitlength=1.mm
	\special{em:linewidth 0.4pt}
	\linethickness{0.4pt}
	\begin{picture}(30.00,27.00)
	\put(24.7,10.65){\vector(1,-2){3.00}}
	\put(21.7,7.65){\makebox(0,0)[rc]{\tiny$1$}}
	\put(24.7,10.65){\vector(-1,-2){3.00}}
	\put(27.7,7.65){\makebox(0,0)[lc]{\tiny$2$}}	
	\put(19.7,15.65){\vector(1,-1){5.00}}
	\put(21.7,12.0){\makebox(0,0)[rc]{\tiny$1$}}
	
	\put(29.7,15.65){\vector(-1,-1){5.00}}
	\put(27.7,12.0){\makebox(0,0)[lc]{\tiny$2$}}
	\qbezier(19.7,15.65)(24.7,19.65)(29.7,15.65)
	\put(29.7,15.65){\vector(2,-1){0.00}}
	\put(25.3,18.965){\makebox(0,0)[rc]{\tiny$2$}}
	
	\qbezier(19.7,15.65)(24.7,11.65)(29.7,15.65)
	\put(19.7,15.65){\vector(-2,1){0.00}}
	\put(25.3,14.965){\makebox(0,0)[rc]{\tiny$1$}}
	\end{picture}
}
+
\hspace{-22mm}
\raisebox{-30pt}[10pt][10pt]{
	\unitlength=1.mm
	\special{em:linewidth 0.4pt}
	\linethickness{0.4pt}
	\begin{picture}(30.00,27.00)
	\put(24.7,10.65){\vector(1,-2){3.00}}
	\put(21.7,7.65){\makebox(0,0)[rc]{\tiny$1$}}
	\put(24.7,10.65){\vector(-1,-2){3.00}}
	\put(27.7,7.65){\makebox(0,0)[lc]{\tiny$2$}}	
	\put(19.7,15.65){\vector(1,-1){5.00}}
	\put(21.7,12.0){\makebox(0,0)[rc]{\tiny$2$}}
	
	\put(29.7,15.65){\vector(-1,-1){5.00}}
	\put(27.7,12.0){\makebox(0,0)[lc]{\tiny$2$}}
	\qbezier(19.7,15.65)(24.7,19.65)(29.7,15.65)
	\put(29.7,15.65){\vector(2,-1){0.00}}
	\put(25.3,18.965){\makebox(0,0)[rc]{\tiny$1$}}
	
	\qbezier(19.7,15.65)(24.7,11.65)(29.7,15.65)
	\put(19.7,15.65){\vector(-2,1){0.00}}
	\put(25.3,14.965){\makebox(0,0)[rc]{\tiny$1$}}
	\end{picture}
}
+c\cdot
\hspace{-24mm}
\raisebox{-20pt}[10pt][10pt]{
	\unitlength=1.mm
	\special{em:linewidth 0.4pt}
	\linethickness{0.4pt}
	\begin{picture}(30.00,27.00)
	\put(24.7,10.65){\vector(1,-2){3.00}}
	\put(21.7,7.65){\makebox(0,0)[rc]{\tiny$1$}}
	\put(24.7,10.65){\vector(-1,-2){3.00}}
	\put(27.7,7.65){\makebox(0,0)[lc]{\tiny$2$}}	
	\end{picture}
}.
$$
The five graphs with fixed values either $1$ or $2$ can be obtained from the two labelled graphs below, by letting the labels in $(i,j,k,l,m,n)$ and $(i,j)$ all run over the values $1$ and $2$ in a sum.
$$\cH=\frac{1}{2} \left(
\hspace{-20mm}
\raisebox{-30pt}[25pt][14pt]{
	\unitlength=1.mm
	\special{em:linewidth 0.4pt}
	\linethickness{0.4pt}
	\begin{picture}(30.00,27.00)
	\put(24.7,10.65){\vector(1,-2){3.00}}
	\put(21.7,7.65){\makebox(0,0)[rc]{\tiny$i$}}
	\put(24.7,10.65){\vector(-1,-2){3.00}}
	\put(27.7,7.65){\makebox(0,0)[lc]{\tiny$j$}}	
	\put(19.7,15.65){\vector(1,-1){5.00}}
	\put(21.7,12.0){\makebox(0,0)[rc]{\tiny$k$}}
	\put(29.7,15.65){\vector(-1,-1){5.00}}
	\put(27.7,12.0){\makebox(0,0)[lc]{\tiny$m$}}
	\qbezier(19.7,15.65)(24.7,19.65)(29.7,15.65)
	\put(29.7,15.65){\vector(2,-1){0.00}}
	\put(25.3,18.965){\makebox(0,0)[rc]{\tiny$l$}}
	\qbezier(19.7,15.65)(24.7,11.65)(29.7,15.65)
	\put(19.7,15.65){\vector(-2,1){0.00}}
	\put(25.3,14.965){\makebox(0,0)[rc]{\tiny$n$}}
	\end{picture}
}
+c\cdot
\hspace{-25mm}
\raisebox{-20pt}[10pt][10pt]{
	\unitlength=1.mm
	\special{em:linewidth 0.4pt}
	\linethickness{0.4pt}
	\begin{picture}(30.00,27.00)
	\put(24.7,10.65){\vector(1,-2){3.00}}
	\put(21.7,7.65){\makebox(0,0)[rc]{\tiny$i$}}
	\put(24.7,10.65){\vector(-1,-2){3.00}}
	\put(27.7,7.65){\makebox(0,0)[lc]{\tiny$j$}}	
	\end{picture}
}\right).
$$
As a graph, $\cH$ encodes a bi-vector.
In order to obtain $F$ and $G$, respectively, one must differentiate with respect to $x$ and $y$. Observe that there are two edges, labelled $i$ and $j$, falling on ground vertices. Since a one-vector is encoded by a graph with only $1$ ground vertex, we let one of these edges, edge $j$, fall, by the Leibniz rule, on the above three internal vertices. Summation over the index $j$ guarantees differentiation with respect to both $x$ and $y$.
The same goes for the second graph. We have represented the vector field $\cX$ by two Kontsevich graphs, see \eqref{EqXFig}.
\end{proof}
\begin{rem}
The generic vector field described in Proposition~\ref{SolGen} cannot be realized in terms of Kontsevich graphs, due to the presence of terms like $u_{xxyyy}u^{2}$ when $a+b\neq0$. This term would have to be represented by a graph with $5$ edges falling on a single vertex. Since there are only $3$ vertices in total, one vertex would have to send both its edges to the vertex encoding $u_{xxyyy}$. We know however, that graphs containing such `double edges' vanish identically.
\end{rem}
\begin{rem}
The vector field described by the two graphs in \eqref{EqXFig} exists in higher dimensions but no longer solves the conjugation equation $\schouten{\cP,\cX}=\dot{\cP}$. This is verified by 
evaluating the Poisson differential acting on the graphs encoding $\cX$. Comparing the result with the graph realization of the first tetrahedral flow, as described in \cite[Figure 2]{f16}, one observes that the graphs in the Schouten bracket $\schouten{\cP,\cX}$ do not equal those in the flow $\cQ_{1:6}$. Therefore, the graphs in equation~\eqref{EqXFig} cannot be expected to trivialize the flows of all Poisson structures on all manifolds of dimension $n\geqslant3$.
\end{rem}
%
Given any lattice $L$ in $\mathbb{R}^{n}$, as a finitely generated abelian group, there exist a basis for this lattice, denoted by $b_{1},\ldots,b_{m}$ with $m \leqslant n$. By the Gram-Schmidt orthogonalization process, there exists an invertible map between the this basis and an orthogonal basis $o_{1},\ldots, o_{m}$. This orthogonal basis is not necessarily a basis for the lattice $L$, but rather one for a lattice $\widehat{L}$ that is isomorphic to $L$.
\begin{proposition}
Consider a lattice $L$ in $\mathbb{R}^{n}$ and the associated projection under taking the quotient $\pi: \mathbb{R}^{n}\rightarrow\mathbb{R}^{n}/L$. Let $\cP$ be a Poisson bi-vector on the $n$-dimensional affine manifold $\mathbb{R}^{n}/L$, whose coefficients are $L$-periodic Fourier series. Then
\begin{enumerate}
\item the image of the flow $\dot{\cP}$ under projection is a well-defined bi-vector on $M^{n}=\mathbb{R}^{n}/L$, and 
\item for $n=2$, the image under projection of the trivializing vector field $\cX$ is well defined on $M^{n}=\mathbb{R}^{n}/L$.
\end{enumerate}
\end{proposition}
\begin{proof}
The coefficients $\cP^{ij}$ of $\cP$ are L-periodic Fourier series. Products of derivatives of these Fourier series yield new products of trigonometric functions with certain wavenumbers. The expressions for $\dot{\cP}$, $\cQ_{1:6}$ and $\cX$ are all of this form. These products are reduced by trigonometric substitutions, resulting in Fourier series of higher wavenumbers; The wavenumbers of these new expression are no smaller than the original wavenumbers. Therefore, the coefficients of $\dot{\cP}, \cQ_{1:6}$ and $\cX$ are also L-periodic Fourier series.
We conclude that the image under the projection $\pi$ of the flows $\dot{\cP}$ and $\cQ_{1:6}$ and the vector field $\cX$ are well-defined on $M^{n}=\mathbb{R}^{n}/L$. 
\end{proof}
\begin{example}\label{ExL}
Let us consider a flow defined on the 2-torus $\mathbb{T}^{2}$ with periods $\{1,1\}$. We inspect the following 2-dimensional bi-vector $\cP$, well-defined, on that torus:
\begin{align*}
\cP^{12}(x,y)=u(x,y)= \alpha_{1,1}\sin(2\pi x)\cos(2\pi y) ,\ \alpha_{1,1}\in \mathbb{R}.
\end{align*}
We can compute the corresponding flow using formula~\eqref{Flow2D}:
\begin{align*}
\cQ_{1:0}^{12}=&-128\alpha_{1,1}^{4}\pi^{6}\cos(2\pi x)\cos(2\pi y)\sin^{3}(2\pi x)\sin^{3}(2\pi y)\\
&+128\alpha_{1,1}^{4}\pi^{6}\sin(2\pi x)\sin(2\pi y)\cos^{3}(2\pi x)\cos^{3}(2\pi y).
\end{align*}
This expression contains products of sines and cosines, which we reduce using trigonometric substitutions to obtain:
\begin{align*}
\cQ_{1:0}^{12}&=16\alpha_{1,1}^{4}\pi^{6}\left(\sin(8\pi x)\sin(4\pi y)+\sin(4\pi x)\sin(8\pi y)\right),
\end{align*}
Which is again well-defined on our torus.
From a direct computation and trigonometric substitutions, it follows that:
\begin{align*}
H=-16\alpha_{1,1}^{3}{\pi}^{4}\sin(2\pi x)\cos(2\pi y)\left(\sin^{2}(2\pi x)\sin^{2}(2\pi y)\right.\\
+2\cos^{2}(2\pi x)\sin^{2}(2\pi y)
+\cos^{2}(2\pi y)\cos^{2}(2\pi x)\left. \right).
\end{align*}
The vector field $\cX$ is obtained through $F=H_{y}$ and $G=-H_{x}$. The obtained expressions for $F$ and $G$ are again reduced, by trigonometric substitutions, yielding:
\begin{align*}
F&=\phantom{-}8\alpha_{1,1}^{3}\pi^{5}\left(-2\cos(2\pi x)\cos(2\pi y)-3\cos(6\pi x)\cos(2\pi y)+\cos(2\pi x)\cos(6\pi y)  \right),
\end{align*}
\begin{align*}
G&=-8\alpha_{1,1}^{3}\pi^{5}\left(\phantom{-}2\sin(2\pi x)\sin(2\pi y)-3\sin(2\pi x)\sin(6\pi y)+\sin(6\pi x)\sin(2\pi y)\right).
\end{align*}
We now verify that this vector field $\cX$ on $\mathbb{R}^{2}$ obtained from the first graph in the above Lemma is well-defined on $\mathbb{T}^{2}$.
\end{example}
\subsubsection*{Conclusion}
When restricting to $2$-dimensional affine manifolds, only one of the two graphs in the Kontsevich tetrahedral flow has a non-zero contribution. The cocycle condition holds trivially and we proved that the flow is a coboundary by constructing a trivializing vector field. We also showed how the divergence free part of the trivializing vector field is realizable in terms of Kontsevich graphs. Finally, we explained why both the flow $\dot{\cP}$ and vector field $\cX$ remain well-defined when taking a quotient over a lattice with respect to which the original Poisson bi-vector was periodic.

\small
\subsubsection*{Acknowledgements}
The author thanks R.\,Buring for cooperation and A.\,V.\,Kiselev for gui\-dan\-ce and constructive criticism. 
The author is grateful to the organizers of SDSP~III conference (3--7 August 2015, CVUT D\v{e}\v{c}\'{\i}n, Czech Republic) for stimulating discussions. This research was supported in part by the Graduate School of Science (RuG).

\appendix
\section{Vanishing of $\Gamma_{2}$ in 2D}\label{App}
\noindent
To show that $\frac{1}{2}\left(\Gamma_{2}^{12}(\cP)-\Gamma_{2}^{21}(\cP)\right)=0$ holds in the $2$-dimensional case, we expand the sum in the formula of $\Gamma_{2}^{ij}$ for $i,j,k,\ell,m,k',\ell',m'\in \{1,2\}$:\\
\tiny
\hspace{-10mm}
\begin{minipage}{0.5\textwidth}
\begin{align*}
\Gamma_{2}^{12}(\cP) &= \sum\limits_{j,k,\ell,k',\ell',m'=1}^{2} \cP^{1j}_{k\ell} \cP^{k2}_{k'\ell'} \cP^{k'\ell}_{m'} \cP^{m'\ell'}_{j}\\
&= \sum\limits_{\ell,k',\ell',m'=1}^{2} \cP^{12}_{1\ell} \cP^{12}_{k'\ell'} \cP^{k'\ell}_{m'} \cP^{m'\ell'}_{2}\\
&= \sum\limits_{k',\ell',m'=1}^{2} \cP^{12}_{11} \cP^{12}_{k'\ell'} \cP^{k'1}_{m'} \cP^{m'\ell'}_{2}+\cP^{12}_{12} \cP^{12}_{k'\ell'} \cP^{k'2}_{m'} \cP^{m'\ell'}_{2}\\
&= \sum\limits_{\ell',m'=1}^{2} \cP^{12}_{11} \cP^{12}_{2\ell'} \cP^{21}_{m'} \cP^{m'\ell'}_{2}+\cP^{12}_{12} \cP^{12}_{1\ell'} \cP^{12}_{m'} \cP^{m'\ell'}_{2}\\
&= \cP^{12}_{11} \cP^{12}_{22} \cP^{21}_{1} \cP^{12}_{2}+\cP^{12}_{11} \cP^{12}_{21} \cP^{21}_{2} \cP^{21}_{2}\\
&\phantom{=}\ +\cP^{12}_{12} \cP^{12}_{12} \cP^{12}_{1} \cP^{12}_{2}+\cP^{12}_{12} \cP^{12}_{11} \cP^{12}_{2} \cP^{21}_{2}\\
&= -u_{xx}u_{yy}u_{x}u_{y}+u_{xx}u_{xy}u_{y}^{2}+u_{xy}^{2}u_{x}u_{y}-u_{xy}u_{xx}u_{y}^{2}\\
&= -u_{xx}u_{yy}u_{x}u_{y}+u_{xy}^{2}u_{x}u_{y},
\end{align*}
\end{minipage}
\hspace{5mm}
\begin{minipage}{0.5\textwidth}
\begin{align*}
\Gamma_{2}^{21}(\cP) &= \sum\limits_{j,k,\ell,k',\ell',m'=1}^{2} \cP^{2j}_{k\ell} \cP^{k1}_{k'\ell'} \cP^{k'\ell}_{m'} \cP^{m'\ell'}_{j}\\
&= \sum\limits_{\ell,k',\ell',m'=1}^{2} \cP^{21}_{2\ell} \cP^{21}_{k'\ell'} \cP^{k'\ell}_{m'} \cP^{m'\ell'}_{1}\\
&= \sum\limits_{k',\ell',m'=1}^{2} \cP^{21}_{21} \cP^{21}_{k'\ell'} \cP^{k'1}_{m'} \cP^{m'\ell'}_{1}+\cP^{21}_{22} \cP^{21}_{k'\ell'} \cP^{k'2}_{m'} \cP^{m'\ell'}_{1}\\
&= \sum\limits_{\ell',m'=1}^{2} \cP^{21}_{21} \cP^{21}_{2\ell'} \cP^{21}_{m'} \cP^{m'\ell'}_{1}+\cP^{21}_{22} \cP^{21}_{1\ell'} \cP^{12}_{m'} \cP^{m'\ell'}_{1}\\
&= \cP^{21}_{21} \cP^{21}_{22} \cP^{21}_{1} \cP^{12}_{1}+\cP^{21}_{21} \cP^{21}_{21} \cP^{21}_{2} \cP^{21}_{1}\\
&\phantom{=}\ +\cP^{21}_{22} \cP^{21}_{12} \cP^{12}_{1} \cP^{12}_{1}+\cP^{21}_{22} \cP^{21}_{11} \cP^{12}_{2} \cP^{21}_{1}\\
&=-u_{xy}u_{yy}u_{x}^{2}+u_{xy}^{2}u_{x}u_{y}+u_{xy}u_{yy}u_{x}^{2}-u_{xx}u_{yy}u_{x}u_{y}\\
&= -u_{xx}u_{yy}u_{x}u_{y}+u_{xy}^{2}u_{x}u_{y}.
\end{align*}
\end{minipage}
\normalsize
\end{document}